\documentclass[11pt]{amsart}
\usepackage{xcolor,mathtools}

\include{package}
\definecolor{grn}{rgb}{0,0.6,0}
\definecolor{mrn}{rgb}{0.3,0,0}
\definecolor{blue}{rgb}{0,0,0.7}
\definecolor{Mygray}{rgb}{0.75,0.75,0.75}
\definecolor{auburn}{rgb}{0.43, 0.21, 0.1}
\definecolor{britishracinggreen}{rgb}{0.0, 0.26, 0.15}
\definecolor{taupe}{rgb}{0.28, 0.24, 0.2}
\usepackage{amsmath,amssymb}
\newtheorem{theorem}{Theorem}[section]

\newtheorem{propn}{Proposition}[section]

\newtheorem{lemma}{Lemma}[section]
\newtheorem{defn}{Definition}[section]

\newtheorem{rmk}{Remark}[section]

\usepackage{latexsym,amssymb}
\usepackage{amssymb}
\usepackage{graphicx}
\usepackage{ textcomp }
\usepackage[left=2cm,right=2cm,top=2cm,bottom=2cm]{geometry}

\begin{document}
\baselineskip=14.5pt
\title[Consecutive number fields of small degree with large P\'{o}lya groups]{Large P\'{o}lya groups in simplest cubic fields and consecutive bi-qaudratic fields}

\author{Md. Imdadul Islam, Jaitra Chattopadhyay and Debopam Chakraborty}
\address[Md. Imdadul Islam and Debopam Chakraborty]{Department of Mathematics, BITS-Pilani, Hyderabad campus, Hyderabad, INDIA}
\address[Jaitra Chattopadhyay]{Department of Mathematics, Siksha Bhavana, Visva-Bharati, Santiniketan - 731235, West Bengal, India}

\email[Md. Imdadul Islam]{p20200059@hyderabad.bits-pilani.ac.in}

\email[Jaitra Chattopadhyay]{jaitra.chattopadhyay@visva-bharati.ac.in; chat.jaitra@gmail.com} 

\email[Debopam Chakraborty]{debopam@hyderabad.bits-pilani.ac.in}

\begin{abstract}
The P\'{o}lya group of an algebraic number field is the subgroup generated by the ideal classes of the products of prime ideals of equal norm inside the ideal class group. Inspired by a recent work on consecutive quadratic fields with large class numbers by Cherubini et al., we extend the notion of {\it consecutiveness} of number fields to certain parametric families of cyclic cubic fields and bi-quadratic fields and address the question of the existence of infinitely many such consecutive fields with large P\'{o}lya groups. This extends a recent result of the second author and Saikia for totally real bi-quadratic fields. 
\end{abstract}

\renewcommand{\thefootnote}{}



\footnote{\emph{We confirm that all the data are included in the article.}}

\renewcommand{\thefootnote}{\arabic{footnote}}
\setcounter{footnote}{0}

\maketitle

\section{Introduction}

Let $K$ be an algebraic number field with ring of integers $\mathcal{O}_{K}$, discriminant $d_{K}$, ideal class group $Cl_{K}$ and class number $h_{K}$. There has been much study on the divisibility and indivisibility properties of class numbers of quadratic fields after Ankeny and Chowla's \cite{AC} seminal work on the existence of infinitely many imaginary quadratic fields with class numbers divisible by any given integer. Inspired by a recent work of Iizuka \cite{iizuka}, similar questions for the class numbers of {\it consecutive quadratic fields} gained a lot of attention. In the present article, we consider a few allied questions regarding the {\it P\'{o}lya group} $Po(K)$ of $K$ whose definition we recall as follows. 

\begin{defn} (cf. \cite{cahen-chabert-book}, \textsection II.4, \cite{polya})
The P\'{o}lya group of $K$ is defined to be the subgroup generated by the ideal classes $[\Pi_{q}(K)]$ inside $Cl_{K}$, where 
\begin{equation}
\displaystyle\Pi_{q}(K) = \left\{\begin{array}{ll}
\displaystyle\prod_{\substack {\mathfrak{p} \in {\rm{Spec}}(\mathcal{O}_{K})\\ N_{K/\mathbb{Q}} \mathfrak{p} = q}}\mathfrak{p} &\mbox{; if } \mathcal{O}_{K} \mbox{ has primes } \mathfrak{p} \mbox{ of norm } q,\\
\mathcal{O}_{K} &\mbox{; otherwise}.
\end{array}\right.
\end{equation}
\end{defn}

The number field $K$ is said to be a P\'{o}lya field if $Po(K)$ is trivial. Unlike the notoriously difficult problem of the determination of number fields $K$ with $h_{K} = 1$, the characterization of P\'{o}lya fields of small degree is not elusive. Quadratic P\'{o}lya fields have been characterized by Zantema in \cite{zantema} and much later, Leriche \cite[Proposition 3.2]{leriche} classified the cyclic cubic P\'{o}lya fields. In the same paper, Leriche classified cyclic quartic P\'{o}lya fields \cite[Theorem 4.4]{leriche} and bi-quadratic P\'{o}lya fields that are compositum of two quadratic P\'{o}lya fields \cite[Theorem 5.1]{leriche}.

\smallskip

Even though there are plenty of number fields $K$ with trivial $Po(K)$, the size of $Po(K)$ can be arbitrarily large as well. In fact, from \cite[Proposition 1.4]{cahen-chabert-book}, this immediately follows for quadratic fields. The analogous result for totally real bi-quadratic fields have been proven by the second author and Saikia in \cite{self-rnt}. Motivated by the work in \cite{imrn}, they recently considered consecutive quadratic fields with large P\'{o}lya groups (cf. \cite[Theorem 5.4]{self-survey}).

\smallskip

There have been quite a lot of works on the behaviour of $Po(K)$ when $K$ is a bi-quadratic field (cf. \cite{self-rnt}, \cite{self-arxiv}, \cite{rajaei-jnt}, \cite{rajaei}, \cite{maref-2}, \cite{maref-raja}, \cite{recent-jnt}). In these works, mostly bi-quadratic fields $K$ with $Po(K) \simeq \mathbb{Z}/2\mathbb{Z}$ have been shown to exist. In this article, we consider two bi-quadratic fields at a time and study the nature of their P\'{o}lya groups. For non-square integers $m$ and $n$, we let $K_{m,n} := \mathbb{Q}(\sqrt{m},\sqrt{n})$. We call two bi-quadratic fields {\it consecutive} if they are of the form $K_{m,n}$ and $K_{m,n + 1}$ for some integers $m$ and $n$. Our result on consecutive bi-quadratic fields is the following.

\begin{theorem}\label{bq-dji}
Let $p$ and $q$ be two odd primes such that $p = 2q + 1$ and let $r \geq 1$ be an even integer. Then there exist infinitely many consecutive bi-quadratic fields $K_{m,p - 1}$ and $K_{m,p}$ such that $Po(K_{m,p - 1}) \simeq Po(K_{m,p}) \simeq (\mathbb{Z}/2\mathbb{Z})^{r}$. 
\end{theorem}

\begin{rmk}
Our method of proof yields infinitely many choices of the integer $m$ to prove the infinitude of the family ascertained in Theorem \ref{bq-dji}. Since primes of the form $p$ and $q = 2p + 1$, widely known as {\it Sophie-Germain primes}, are conjecturally known to exist in infinite number, Theorem \ref{bq-dji} also conjecturally provides infinitely many pairs $(K_{m,p - 1}, K_{m,p})$ for a fixed value of $m$.
\end{rmk}
Surprisingly, there are not many results in the literature concerning the P\'{o}lya groups of cubic fields. However, some parametric families of cyclic cubic fields have been well studied (cf. \cite{balady}). Among all these, Shank's family of simplest cubic fields \cite{shanks} is the most studied and is defined by the splitting field $K_{n}$ of the irreducible polynomial $f_{n}(X) = X^{3} + (n + 3)X^{2} + nX - 1$ over $\mathbb{Q}$. Our result on large P\'{o}lya groups for this family is the following.
\begin{theorem}\label{main-cubic}
Let $M > 0$ be a real number. Then there exist infinitely many simplest cubic fields $K_{n}$ such that $|Po(K_{n})| > M$.
\end{theorem}
\begin{rmk}
Since the P\'{o}lya group of a number field is a subgroup of the ideal class group, it readily follows from Theorem \ref{main-cubic} that Shank's family of simplest cubic fields contains infinitely many fields whose class groups have arbitrarily large $3$-ranks.
\end{rmk}
\begin{rmk}
In \cite{balady}, Balady considered a general class of cyclic cubic fields and proved that Shank's family can be obtained as a particular case of the general family. The techniques of proving Theorem \ref{main-cubic} will work for other parametric families of cyclic cubic fields mentioned in \cite{balady} as well.
\end{rmk}
\begin{rmk}
We can declare two simplest cubic fields to be consecutive if they are of the form $K_{n - 1}$ and $K_{n}$ for some integer $n \geq 2$. Our method to prove Theorem \ref{main-cubic} does not allow us to conclude anything regarding large P\'{o}lya groups of consecutive cubic fields. We shall make use of the density of primes in an arithmetic progression such that certain quadratic polynomials produce square-free values at those primes. In most cases, the density turns out to be smaller than $\frac{1}{2}$ and therefore we are unable to infer whether there exist a common set of primes for which more than one quadratic polynomial produce square-free values simultaneously. 
\end{rmk}

\section{Preliminaries}

We shall apply Zantema's result concerning the P\'{o}lya group of $K$ and the ramified primes in $K/\mathbb{Q}$ when $K$ is a finite Galois extension of $\mathbb{Q}$. The Galois group $G := {\rm{Gal}}(K/\mathbb{Q})$ naturally acts on the unit group $\mathcal{O}_{K}^{*}$ and thus makes it a $G$-module. Zantema established a relation between $Po(K)$ and the cohomology group $H^{1}(G,\mathcal{O}_{K}^{*})$ as follows.  

\begin{theorem} \cite[Page 163]{zantema}\label{zmt}
Let $K/\mathbb{Q}$ be a finite Galois extension with Galois group $G := {\rm{Gal}}(K/\mathbb{Q})$ and let $e_{1},\ldots,e_{s}$ be the ramification indices of the ramified primes in $K/\mathbb{Q}$. Then there exists an exact sequence of abelian groups as follows.
\begin{equation}\label{exact-equn}
0 \to H^{1}(G,\mathcal{O}_{K}^{*}) \to \displaystyle\bigoplus_{i = 1}^{s}\mathbb{Z}/e_{i}\mathbb{Z} \to Po(K) \to 0.
\end{equation}
\end{theorem}

The next lemma, which is similar in nature to \cite[Lemma 2]{self-rnt}, enables us to find infinitely many prime numbers satisfying certain congruence and Legendre symbol conditions.

\begin{lemma}\label{main-CRT-lemma}
Let $t \geq 2$ be an integer and let $p$ and $q$ be given odd prime numbers. Then there exist infinitely many $t$-tuples of prime numbers $\{r_{1},\ldots,r_{t}\}$ such that $r_{i} \equiv 1 \pmod {8pq}$ for all $i \in \{1,\ldots,t\}$ and $\left(\dfrac{r_{i}}{r_{j}}\right) = - 1$ for all $i$ and $j$ with $i \neq j$.
\end{lemma}
\begin{proof}
We prove the lemma by induction on $t$. For $t = 2$, we first choose a prime number $r_{1} \equiv 1 \pmod {8pq}$. Indeed, such a choice is possible due to Dirichlet's theorem for primes in arithmetic progressions. Then we choose an integer $n_{1}$ with $1 \leq n_{1} \leq r_{1} - 1$ and $\left(\dfrac{n_{1}}{r_{1}}\right) = -1$. By Chinese remainder theorem, there exists a unique integer $x_{0} \pmod {8pqr_{1}}$ that satisfies the congruences $x_{0} \equiv 1 \pmod {8pq}$ and $x_{0} \equiv n_{1} \pmod {r_{1}}$ simultaneously. Then we have $\gcd(x_{0},8pqr_{1}) = 1$ and hence again by Dirichlet's theorem, there exist infinitely many primes $r_{2}$ such that $r_{2} \equiv x_{0} \pmod {8pqr_{1}}$. Consequently, we have $r_{2} \equiv x_{0} \equiv 1 \pmod {8}$ and $\left(\dfrac{r_{2}}{r_{1}}\right) = \left(\dfrac{x_{0}}{r_{1}}\right) = \left(\dfrac{n_{1}}{r_{1}}\right) = -1$.

\smallskip

Now, we assume that $r_{1},\ldots,r_{t - 1}$ are prime numbers with $r_{i} \equiv 1 \pmod {8pq}$ and $\left(\dfrac{r_{i}}{r_{j}}\right) = -1$ for all $i,j \in \{1,\ldots,t - 1\}$ and $i \neq j$. For each $i \in \{1,\ldots,t - 1\}$, we choose integers $n_{i}$ such that $1 \leq n_{i} \leq r_{i} - 1$ and $\left(\dfrac{n_{i}}{r_{i}}\right) = -1$. Consequently, by Chinese remainder theorem, there exists a unique integer $x_{0} \pmod {8pqr_{1}\cdots r_{t - 1}}$ satisfying $x_{0} \equiv 1 \pmod {8pq}$ and $x_{0} \equiv n_{i} \pmod {r_{i}}$ for each $i$. Again, by Dirichlet's theorem, there exist infinitely many prime numbers $r_{t}$ such that $r_{t} \equiv x_{0} \pmod {8pqr_{1}\cdots r_{t - 1}}$. Hence the $t$-tuple $\{r_{1},\ldots,r_{t}\}$ of prime numbers satisfies the desired conditions of the lemma and the proof is thus complete.
\end{proof}

We quote the next proposition from \cite{leriche} which relates the size of the P\'{o}lya group of a cycic extension of $\mathbb{Q}$ to the number of ramified primes.
\begin{propn} \cite[Proposition 2.3]{leriche}
Let $q \geq 3$ be prime number and let $K/\mathbb{Q}$ be a cyclic extension of degree $q$. Then $|Po(K)| = q^{r_{K}} - 1$, where $r_{K}$ stands for the number of ramified primes in $K/\mathbb{Q}$.
\end{propn}

We now recall a few results that facilitate us in understanding the group $H^{1}(G,\mathcal{O}_{K}^{*})$ for a bi-quadratic field $K$. We first fix certain notations. For an integer $m \neq 0$, we denote by $[m]$ its canonical image in the group $\mathbb{Q}^{*}/(\mathbb{Q}^{*})^{2}$. For integers $m_{1},\ldots,m_{r}$, we denote the subgroup generated by $[m_{1}],\ldots,[m_{r}]$ in $\mathbb{Q}^{*}/(\mathbb{Q}^{*})^{2}$ by $\langle m_{1},\ldots,m_{r} \rangle$.

\begin{lemma} \cite[Theorem 4]{setzer} \label{set-zer}
For a bi-quadratic field $K$ with quadratic fields $K_{1}, K_{2}$ and $K_{3}$, if the rational prime $2$ is not totally ramified in $K/\mathbb{Q}$, then $H^{1}(G,\mathcal{O}_{K}^{*})$ is same as its $2$-torsion subgroup $H^{1}(G,\mathcal{O}_{K}^{*})[2]$.
\end{lemma}

\begin{lemma} \cite[Lemma 4.3]{zantema}\label{six-lemma}
Let $K$ be a bi-quadratic field with quadratic subfields $K_{1}, K_{2}$ and $K_{3}$. For $i = 1,2,3$, let $d_{K_{i}}$ be the discriminant of $K_{i}$ and let $u_{i} = z_{i} + t_{i}d_{K_{i}}$ be a fundamental unit of $\mathcal{O}_{K_{i}}$ where $z_{i} > 0$. Then $H^{1}(G,\mathcal{O}_{K}^{*})[2]$ is isomorphic to the subgroup $\langle d_{1},d_{2},d_{3},a_{1},a_{2},a_{3} \rangle$, where $a_{i} = N_{K_{i}/\mathbb{Q}}(u_{i} + 1)$ if $N_{K_{i}/\mathbb{Q}}(u_{i}) = 1$ and $1$ otherwise.
\end{lemma}

In the notation of Lemma \ref{six-lemma}, if $N_{K_{i}/\mathbb{Q}}(u_{i}) = z_{i}^{2} - t_{i}^{2}d = 1$, then we have $$N_{K_{i}/\mathbb{Q}}(u_{i} + 1) = (z_{i} + 1)^{2} - t_{i}^{2}d = z_{i}^{2} - t_{i}^{2}d + 2z_{i} + 1 = 2(z_{i} + 1).$$ 

\begin{lemma} \cite[Lemma 2.1]{rajaei}\label{raja-lemma}
Let $p \geq 3$ be a prime number such that $p \equiv 3 \pmod {4}$ and let $u = z + t\sqrt{p}$ be a fundamental unit of $\mathbb{Q}(\sqrt{p})$. Then we have
\begin{equation*}
N_{\mathbb{Q}(\sqrt{p})/\mathbb{Q}}(u + 1)=
\begin{cases}
    2p   & ~ \text{ if } p \equiv 3 \pmod {8},\\
    2   & ~ \text{ if } p \equiv 7 \pmod {8}.
\end{cases}
\end{equation*}
\end{lemma}

We observe that to apply Lemma \ref{six-lemma}, it is required to know the sign of the fundamental unit of the suitable quadratic field. Then next lemma addresses this issue for a particular class of quadratic fields.

\begin{lemma} \cite{trotter})\label{trott-lemma}
Let $t \geq 3$ be an odd integer and let $p_{1},\ldots,p_{t}$ be prime numbers with $p_{i} \equiv 1 \pmod {4}$ and $\left(\dfrac{p_{i}}{p_{j}}\right) = -1$ for each $i \neq j$. The the fundamental unit of the quadratic field $\mathbb{Q}(\sqrt{p_{1}\cdots p_{t}})$ is $-1$.
\end{lemma} 
To prove Theorem \ref{main-cubic}, we make use of the following proposition which is taken from \cite{leriche}.
\begin{propn} \cite[Proposition 2.6]{leriche}\label{cubic-wala-proposition}
Let $q \geq 3$ be a prime number and let $K$ be a cyclic extension of $\mathbb{Q}$ of degree $q$. If $r_{K}$ is the number of ramified primes in $K/\mathbb{Q}$, then $Po(K) = q^{r_{K}} - 1$.
\end{propn}
To apply Proposition \ref{cubic-wala-proposition} in our context of Theorem \ref{main-cubic}, we need to ensure that the discriminant of $K_{n}$ has a large number of prime divisors. Now, the discriminant of $K_{n}$ is known to be $(n^{2} + 3n + 9)^{2}$, provided the polynomial $h(n) = n^{2} + 3n + 9$ is square-free. Therefore, we need to know the distribution of the integers $n$ such that $h(n)$ is square-free. Thus we are required to ensure that $h(n)$ is square-free as well as has large number of prime divisors for infinitely many values of $n$. Our strategy is to employ Dirichlet's theorem for primes in an arithmetic progression to catch hold of infinitely many primes and to show that $h$ satisfies the above two conditions for those primes. To achieve this, we now briefly discuss the square-free values of $h$ where the arguments are primes in an arithmetic progression. This is kindly shared with us by Prof. H. Pasten through \cite{pasten-private}, which is broadly based on an earlier work by him \cite{pasten-ijnt}. We furnish the outline of his method for the polynomial $h$.

\smallskip

For the sake of brevity, we shall assume that $p$ and $q$ stand for prime numbers in rest of the discussion in this section. For a pair of relatively prime positive integers $a$ and $m$ where $m$ is square-free, a large positive real number $X$ and the quadratic polynomial $h(X) = X^{2} + 3X + 9$, let $$S_{h}(X) = \{p \leq X : p \equiv a \pmod {m} \mbox{ and } f(p) \mbox{ is square-free} \}.$$ We denote the cardinality of the set $S_{h}(X)$ by $N_{h}(X;a,m)$. For a fixed small real number $\tau > 0$, let $y = \frac{\log X}{100}$ and let $z = X^{1 - \tau}$. We define the following sets similar to what has been defined in \cite{pasten-ijnt}. $$Q = \{p \leq X : p \equiv a \pmod {m} \mbox{ and } f(p) \not\equiv 0 \pmod {q^{2}} \mbox{ for all } q \leq y\},$$ $$R = \{p \leq X : q^{2} \mid f(p) \mbox{ for some } q \mbox{ with } y < q \leq z\},$$ $$S = \{p \leq X : q^{2} \mid f(p) \mbox{ for some } q \mbox{ with } q > z\}.$$ It is clear that $|Q| \geq N_{h}(X;a,m)$. Moreover, if $p \in Q$, then either $f(p)$ is square-free or $f(p)$ is divisible by the square of some prime number that lies either in $(y,z]$ or in $(z,\infty)$. In other words, $p \in S_{h}(X) \cup R \cup S$. It then follows that $|Q| \leq N_{h}(X;a,m) + |R| + |S|$. Consequently, $|Q| \geq N_{h}(X;a,m) \geq |Q| - |R| - |S|$. By \cite[Lemma 2.6]{pasten-ijnt}, we have $|R| = o\left(\frac{X}{\log X}\right)$ and an argument similar to (1.5) in \cite{fried} yields $|S| = o\left(\frac{X}{\log X}\right)$. Thus to derive an estimate for $N_{h}(X;a,m)$ amounts to derive an estimate for $|Q|$.

\smallskip

Let $\rho(n)$ denote the cardinality of the set $\{b \in (\mathbb{Z}/n\mathbb{Z})^{*} : h(b) \equiv 0 \pmod {n} \mbox{ and } b \equiv a \pmod {\gcd(m,n)}\}$. By Chinese remainder theorem, we observe that $\rho(n)$ is a multiplicative function of $n$. Moreover, an application of Hensel's lemma (cf. \cite[Lemma 5.2]{murty-pasten}) yields that $\rho(q^{2}) \leq 2$ for all prime $q \neq 3$ as $3$ is the only prime divisor of the discriminant of $h$ and it turns out by a direct computation that $\rho(9) = 0$.  

\smallskip

Now, we define the Euler product $c_{h}(m,a) := \displaystyle\prod_{q}\left(1 - \frac{\rho(q^{2})\phi(\gcd(m,q^{2}))}{\phi(q^{2})}\right)$. Since $m$ is assumed to be square-free, we see that $\gcd(m,q^{2}) = 1$ or $q$ and the gcd equals $q$ precisely for those primes $q$ that divides $m$. Therefore, for all but finitely many primes $q$, we see that each term of the Euler product is $1 - \frac{\rho(q^{2})}{\phi(q^{2})}$. Using $\rho(q^{2}) \leq 2$, we conclude that $c_{h}(m,a)$ is convergent. Moreover, since each term in the Euler product is non-zero, we conclude that $c_{h}(m,a) \neq 0$. Now, an argument similar to the one used in \cite{pasten-ijnt} gives that $|Q| = \frac{c_{h}(m,a)}{\phi(m)} \cdot \frac{X}{\log X} + o\left(\frac{X}{\log X}\right)$ and therefore, we have 
\begin{equation}\label{pasten-pasten-pasten}
N_{h}(X;a,m) = \frac{c_{h}(m,a)}{\phi(m)} \cdot \frac{X}{\log X} + o\left(\frac{X}{\log X}\right).
\end{equation} In particular, we can conclude that $S_{h}(X)$ is an infinite set because $c_{h}(m,a) > 0$.

\section{Proof of Theorem \ref{bq-dji}}

Let $p$ and $q$ be odd prime numbers with $p = 2q + 1$ so that $p \equiv 3 \pmod {4}$. Let $t \geq 3$ be an odd integer and let $r_{1},\ldots,r_{t}$ be primes numbers satisfying the hypotheses of Lemma \ref{main-CRT-lemma}. Let $m = r_{1}\cdots r_{t}$ and $K_{m,p} := \mathbb{Q}(\sqrt{m},\sqrt{p})$. Then the quadratic subfields of the bi-quadratic field $K_{m,p}$ are $K_{m} := \mathbb{Q}(\sqrt{m})$, $K_{p} := \mathbb{Q}(\sqrt{p})$ and $K_{mp} := \mathbb{Q}(\sqrt{mp})$. We note that the ramified primes in $K/\mathbb{Q}$ are precisely $2,p,r_{1},\ldots,r_{t}$, each with ramification index $2$. Therefore, by Theorem \ref{zmt}, we have $Po(K_{m,p}) \simeq \displaystyle\bigoplus_{i = 1}^{t + 2}\mathbb{Z}/2\mathbb{Z}\Bigg/ H^{1}(G,\mathcal{O}_{K_{m,p}}^{*})$. We also note that since the rational prime $2$ is not totally ramified in $K/\mathbb{Q}$, therefore by Lemma \ref{set-zer}, the cohomology group is equal to its $2$-torsion subgroup.

\smallskip

In the notations of Lemma \ref{six-lemma}, we have $d_{K_{m}} = m$, $d_{K_{p}} = 4p$ and $d_{K_{mp}} = 4mp$. Therefore, $[d_{K_{m}}]$, $[d_{K_{p}}]$, $[d_{K_{mp}}] \in \langle [p],[mp] \rangle$ in $\mathbb{Q}^{*}/(\mathbb{Q}^{*})^{2}$. Now, by Lemma \ref{trott-lemma}, the fundamental unit of $K_{m}$ has norm $-1$ and therefore by Lemma \ref{six-lemma}, we have $a_{K_{m}} = 1$. Also, by Lemma \ref{raja-lemma}, we conclude that $[a_{K_{p}}] = [2]$ or $[2p]$ in $\mathbb{Q}^{*}/(\mathbb{Q}^{*})^{2}$.

\smallskip

Let $u = z + t\sqrt{mp}$ be a fundamental unit of $K_{mp}$. If $N_{K_{mp}/\mathbb{Q}}(u) = -1$, then we have by Lemma \ref{six-lemma} that $H^{1}(G,\mathcal{O}_{K_{mp}}^{*}) \simeq \langle [2],[m],[p] \rangle$ in $\mathbb{Q}^{*}/(\mathbb{Q}^{*})^{2}$. We prove that this is also the case when the norm is $1$. Thus we now assume that $N_{K_{mp}/\mathbb{Q}}(u) = 1$. That is, $z^{2} - t^{2}mp = 1$.

\smallskip

\noindent
{\bf Case 1.} $z$ is even and $t$ is odd. Then $\gcd(z - 1,z + 1) = 1$ since these are consecutive odd integers. From $(z - 1)(z + 1) = t^{2}mp$, we obtain $z + 1 = \ell_{1}^{2}x$ and $z - 1 = \ell_{2}^{2}y$ for integers $\ell_{1},\ell_{2},x$ and $y$ such that $\ell_{1}\ell_{2} = t$ and $xy = mp = r_{1}\cdots r_{t}p$. Then from this, we obtain $\ell_{1}^{2}x - \ell_{2}^{2}y = 2$ and therefore $$[a_{mp}] = [2(z + 1)] = [2\ell_{1}^{2}x] = [2x] \in \langle [2],[m],[p] \rangle \mbox{ if and only if } x = 1,m,p,mp.$$ We now prove that no other choice of $x$ is admissible. For otherwise, let $x$ comprise of odd number of $r_{i}'$s and let $r_{j}$ be a prime divisor of $y$. From the equation $\ell_{1}^{2}x - \ell_{2}^{2}y = 2$, we have $$1 = \left(\dfrac{2}{r_{j}}\right) = \left(\dfrac{\ell_{1}^{2}x - \ell_{2}^{2}y}{r_{j}}\right) = \left(\dfrac{x}{r_{j}}\right) = -1,$$ a contradiction. Similarly, if $x$ comprises of an even number of $r_{i}'$s, then $y$ comprises of odd number of $r_{i}'$s as $t$ is odd. We choose a a prime divisor $r_{j}$ of $x$ to obtain $$1 = \left(\dfrac{2}{r_{j}}\right) = \left(\dfrac{\ell_{1}^{2}x - \ell_{2}^{2}y}{r_{j}}\right) = \left(\dfrac{-1}{r_{j}}\right)\left(\dfrac{y}{r_{j}}\right) = -1,$$ a contradiction. Consequently, in this case, we have $H^{1}(G,\mathcal{O}_{K_{mp}}^{*}) \simeq \langle [2],[m],[p] \rangle$ in $\mathbb{Q}^{*}/(\mathbb{Q}^{*})^{2}$.

\smallskip

\noindent
{\bf Case 2.} $z$ is odd and $t$ is even. Then from $z^{2} - t^{2}mp = 1$, we obtain $\dfrac{z - 1}{2}\cdot \dfrac{z + 1}{2} = (t/2)^{2}mp$ where $\dfrac{z - 1}{2}$ and $\dfrac{z + 1}{2}$ are relatively prime. Now following the same type of reasoning as in Case 1 yields that $H^{1}(G,\mathcal{O}_{K_{mp}}^{*}) \simeq \langle [2],[m],[p] \rangle$ in $\mathbb{Q}^{*}/(\mathbb{Q}^{*})^{2}$. Therefore, $H^{1}(G,\mathcal{O}_{K_{m,p}}^{*}) \simeq (\mathbb{Z}/2\mathbb{Z})^{3}$ and hence $Po(K_{m,p}) \simeq \displaystyle\bigoplus_{i = 1}^{t - 1}\mathbb{Z}/2\mathbb{Z}$.

\smallskip

Working with the bi-quadratic field $K_{m,p-1} = K_{m,2q}$, we follow the same line of argument used for $K_{m,p}$ to obtain $Po(K_{m,p - 1}) \simeq \displaystyle\bigoplus_{i = 1}^{t - 1}\mathbb{Z}/2\mathbb{Z}$. Since $t \geq 3$ is an arbitrary odd integer, the proof of Theorem \ref{bq-dji} is thus complete. $\hfill\Box$

\section{Proof of Theorem \ref{main-cubic}}

Let $M > 0$ be a real number and let $t \geq 2$ be an integer such that $3^{t - 1} > M$. Let $p_{1},\ldots,p_{t}$ be odd prime numbers such that $\left(\frac{-3}{p_{i}}\right) = 1$ and let $a_{i}'$s be integers such that $a_{i}^{2} \equiv -3 \pmod {p_{i}}$ for all $i$. Let $b_{i}'$s be integers such that $2b_{i} \equiv 1 \pmod {p_{i}}$ and we consider the following system of congruences. $$X \equiv -3b_{i} + 3a_{i}b_{i} \pmod {p_{i}} \mbox{ for all } i = 1,\ldots,t.$$ By Chinese remainder theorem, there exists a unique integer $x_{0} \pmod{p_{1}\cdots p_{t}}$ satisfying the above system of congruences. Since $\gcd(x_{0},p_{1}\cdots p_{t}) = 1$, by Dirichlet's theorem for primes in an arithmetic progression, we conclude that there are infinitely many primes $p \equiv x_{0} \pmod {p_{1}\cdots p_{t}}$. For each such prime $p$, we see that $$h(p) = p^{2} + 3p + 9 \equiv x_{0}^{2} + 3x_{0} + 9 \equiv 0 \pmod {p_{i}} \mbox{ for each } i.$$ Also, by \eqref{pasten-pasten-pasten}, we conclude that $h(p)$ is square-free for infinitely many such $p'$s. Therefore, the discriminant of the simplest cubic field $K_{p}$ is given by $(p^{2} + 3p + 9)^{2}$ and is divisible by the primes $p_{1},\ldots,p_{t}$. In other words, the primes $p_{1},\ldots,p_{t}$ are all ramified in $K_{p}/\mathbb{Q}$ with ramification index $3$. Consequently, by Proposition \ref{cubic-wala-proposition}, we conclude that $|Po(K_{p})| \geq 3^{t - 1} > M$. This completes the proof of Theorem \ref{main-cubic}. $\hfill\Box$

\bigskip

{\bf Acknowledgements.} It is a huge pleasure to thank Prof. H. Pasten for proving the equality mentioned in \eqref{pasten-pasten-pasten}. We gratefully acknowledge his kind efforts to address this issue and sending us his proof \cite{pasten-private}. The research of the first author is funded by CSIR (File no: 09/1026(0036)/2020-EMR-I).

\end{document}